\documentclass[11pt]{amsart}
\usepackage{amssymb,amsmath,amsfonts,amsthm, mathtools}
\usepackage[colorlinks=true, urlcolor=blue,linkcolor=blue, citecolor=blue]{hyperref}
\usepackage{verbatim}
\usepackage{tikz, xcolor}
\usepackage{appendix}

\def\qed{\hfill \mbox{\rule{0.5em}{0.5em}}}

\newtheorem{theorem}{Theorem}[section]
\newtheorem{corollary}{Corollary}[section]

\newtheorem{proposition}{Proposition}[section]

\newtheorem{example}{Example}[section]
\newtheorem{remark}{Remark}[section]

\begin{document}


\title[\resizebox{4.5in}{!}{Asymptotic results on weakly increasing subsequences in random words}]{Asymptotic results on weakly increasing subsequences in random words}

\author{\"{U}m\.{i}t I\c{s}lak}

\address{%
Bo\u{g}azi\c{c}i University\\
Faculty of Arts and Science \\
Department of Mathematics\\
Bebek-Istanbul, Turkey, 34342}

\email{umit.islak1@boun.edu.tr}

\author{Alperen Y. \"{O}zdemir}
\address{%
University of Southern California\\
Department of Mathematics\\
Los Angeles, California, 90089-2532}
\email{aozdemir@usc.edu}


\date{March 03, 2018}

\begin{abstract} 
Let $X=(X_1,\ldots,X_n)$ be a vector of i.i.d. random variables where
$X_i$'s take values over $\mathbb{N}$. The purpose of this paper is to study the number of weakly
increasing subsequences of $X$ of a given length $k$, and the number of all weakly
increasing subsequences of $X$. For the former, it is
shown that a central limit theorem holds. Also, the first two
moments of each of those two random variables are analyzed, their asymptotics are investigated, and
results are related to the case of similar statistics in uniformly random permutations. We conclude the paper with 
applications on a similarity measure of Steele, and on increasing subsequences of  riffle shuffles.
\end{abstract}

\maketitle

\section{Introduction}\label{sec:intro} 

Let $X_1,X_2,\ldots,$ be a sequence of independent and identically distributed (i.i.d.) random variables whose support is a finite subset of $\mathbb{N}:=\{1,2,\ldots\}$, and set  $p_j = \mathbb{P}(X_1 = j)$, $j \in
\mathbb{N}$. Throughout the paper, we assume that the probability measure $\{p_j\}$ is non-degenerate. The purpose of this work is to
study two statistics related to the sequence $X_1,X_2,\ldots$, the first one being
\begin{equation*} \label{def:Yk}
Y_{n,k}^{\mathbf{(p)}}= \sum_{1 \leq i_1 < \cdots < i_k \leq n} \mathbf{1}
(X_{i_1} \leq \cdots \leq X_{i_k}).
\end{equation*}
In words, $Y_{n,k}^{\mathbf{(p)}}$ is the number of weakly increasing
subsequences of $X_1,\ldots,X_n$ that have length $k$.

The second statistic of interest, which is closely related
to the former, is the total number of weakly increasing
subsequences of $X_1,\ldots,X_n$. That is defined by
\begin{equation*}\label{def:Y}
    Y_n^{\mathbf{(p)}} = \sum_{k = 0}^n Y_{n,k}^{\mathbf{(p)}}.
\end{equation*} 
Here and below, we include the empty subsequence (corresponding to
$k=0$) merely for ease in some computations. One case of special interest is  the case of  uniform random words, where $p_j=1/a$ for $j \in [a]:=\{1,\ldots,a\}$, $a \in
\mathbb{N}$.  

Certain aspects of the number of increasing subsequences problem in
a uniformly random permutation setting are well studied. See \cite{fulman2},
\cite{LP1981} and \cite{pinsky}. In particular, the work by
Lifschitz and Pittel \cite{LP1981} establish the asymptotic  order of the first
two moments of total number of increasing subsequences in a uniformly
random permutation $\pi$. Namely, letting $Z_{n,k}$ be the number of increasing subsequences of $\pi$ of length $k$, and defining   $$Z_n = \sum_{k = 0}^n Z_{n,k}$$ to be the total number of
increasing subsequences in $\pi$, they show
that the first two moments are given by  
\begin{equation*}\label{eqn:EZn}\mathbb{E}[Z_n] =
\sum_{k=0}^n \frac{1}{k!} \binom{n}{k}
\end{equation*} and
\begin{equation*}\label{eqn:VarZn}\mathbb{E}[(Z_n)^2] = \sum_{k+l \leq n} 4^l ((k+l)!)^{-1} \binom{n}{k+l} \binom{(k+1)/2+l-1}{l}.
\end{equation*}
Moreover, they prove that the asymptotic relations
$$\mathbb{E}[Z_n] \sim (2 \sqrt{\pi e} )^{-1} n^{-1/4}
\exp(2n^{1/2}),$$ and
$$\mathbb{E}[(Z_n)^2] \sim cn^{-1/4} \exp\left(2 \sqrt{2 + \sqrt{5}}
n^{1/2}\right)$$ hold as  $n \rightarrow \infty$, where $c \approx 0.0106$. More recently, Pinsky \cite{pinsky} shows that the weak law of large numbers
$$\frac{Z_{n,k}}{\mathbb{E}[Z_{n,k}]} \longrightarrow_\mathbb{P} 1, \quad \text{as} \quad n \rightarrow \infty $$
is satisfied by the sequence $Z_{n,k}$ when $k=o(n^{\frac{2}{5}})$. In a follow-up work \cite{pinsky2}, he shows that the weak law of large numbers fails if $k$ is of order larger than $n^{\frac{4}{9}}.$\\

To the best of authors' knowledge, the moments and asymptotics of
$Y_{n,k}^{\mathbf{(p)}}$ and $Y_{n}^{(\mathbf{p})}$ have not been investigated in the literature
except the special case $Y_{n,2}^{(a)}$ (corresponding to the number of
inversions in random words) for which it is known that a central
limit theorem holds. See \cite{bliem} and \cite{Janson}, for two
different proofs and several interesting interpretations of
$Y_{n,2}^{(a)}$. Here,   we
  focus on the analogue   results of Lifschitz and Pittel, and 
investigate the asymptotics in a random word setting. The discussions in \cite{Janson} suggests that the statistics we study  below may have connections to other 
topics, such as Ferrer diagrams and generalized Galois numbers.

Also, as we shall see below, random word and random permutation
cases are indeed quite related, the main difference being the
possibility of having repeated values in the former case. However,
as the alphabet size increases, certain statistics related to the random word
case behave more and more like the corresponding statistic of a
uniformly random permutation. Results quantifying such connections
will be the content of Theorem \ref{discrepgen} below.

The theoretical results we have are supported with two applications, one on similarity measures on sequences, and the other one on riffle shuffles. In Section \ref{sec:Steele}, we will study a similarity measure first introduced by M. Steele in \cite{Steele} as an alternative to the length of longest common subsequences. Via establishing a relation between common subsequences and increasing subsequences, we will be able to understand Steele's statistic  in the random permutation setting. In Section \ref{sec:riffle},  
we will also show that our results for weakly increasing sequences in random words can  be interpreted in terms of 
increasing subsequences of one other class of random permutations; riffle shuffles. In particular, that will prove that the number of inversions in a possibly biased riffle shuffle is asymptotically normal answering a question of Fulman in \cite{fulman} in a more general setting. This was previously studied in \cite{Islak}, and indeed, it was the motivating and beginning question for our study below. 

Let us now fix some notation for the following sections. From here
on, $Y_{n,k}^{\mathbf{(p)}}$ and $Y_{n}^{(\mathbf{p})}$ denote the number of weakly increasing
subsequences of length $k$ and the total number of weakly increasing
subsequences of a random word $X_1,\ldots,X_n$ where $X_i$'s are i.i.d.
random variables with compactly supported $\mathbf{p} = (p_1,p_2,\ldots)$ such that $p_j = \mathbb{P}(X_1=j)$, $j \in \mathbb{N}$. For the uniform case, these two will be replaced by $Y_{n,k}^{(a)}$ and $Y_n^{(a)}$. We denote
the number of increasing subsequences of length $k$ and the total
number of increasing subsequences of a uniformly random permutation
by $Z_{n,k}$ and $Z_n$, respectively.

Also, $=_d$, $\rightarrow_d$ and
$\rightarrow_{\mathbb{P}}$ are used for equality in distribution,
convergence in distribution and convergence in probability,
respectively. $\mathcal{G}$ denotes a standard normal random
variable, and $C$ is used for constants (which may differ in each
line) that do not depend on any of the parameters. Finally, for two
sequences $a_n, b_n$, we write $a_n \sim b_n$ for $\lim_{n
\rightarrow \infty} a_n / b_n =1$.

The rest of the paper is organized as follows. Section \ref{sec:momentswords} gives exact expressions first two moments of $Y_{n,k}^{(\mathbf{p})}$ and $Y_{n}^{(\mathbf{p})}$. These results are supplemented with a central limit for $Y_{n,k}^{(\mathbf{p})}$ in Section \ref{sec:CLTwords}. Later, we  turn our attention to random permutations, and provide exact expressions for the first moments and a CLT for $Z_{n,k}$ in Section \ref{momentsCLTperms}. Sections \ref{sec:secondmomentinperms} and \ref{sec:momentastmptoticsinwords} are devoted a study of moment asymptotics as $n \rightarrow \infty$ in random permutations and random words, respectively. We compare the behaviors of increasing subsequences in these two structures in Section \ref{sec:comparison}. The paper is concluded in Section \ref{sec:applications} with two connections to a similarity measure of Steele, and to increasing subsequences in riffle shuffles. 

\section{Moments for the random words case}\label{sec:momentswords}

We start by giving exact expressions for the first two moments of  $Y_{n,k}^{\; (\cdot)}$, the number of weakly increasing subsequences of a random word of length $k$. 
 
\begin{theorem}\label{thm:Ykmoment} Suppose $\mathbf{p}$ is supported on $[a].$ \\
(i.) We have 
\begin{equation}\label{EY_kgeneral}
\mathbb{E}[Y_{n,k}^{\mathbf{(p)}}] = \binom{n}{k}
\sum_{\{x_i\} \in S_{a,k}} \left(\prod_{i=1}^a p_i^{x_i}\right),
\end{equation}
where $S_{a,k} = \{(x_1,\ldots,x_a) : x_i \in \mathbb{N} \cup \{0\}, \;
\sum_{i=1}^a x_i = k\},$ the set of non-negative partitions of $k$ into exactly $a$ parts. \\
(ii.)  $\mathbb{E}[Y_{n,k}^{\mathbf{(p)}}]$
is maximized for the uniform distribution over $[a]$ for which  
\begin{multline}\label{EYk}
    \mathbb{E}[Y_{n,k}^{(a)}] = \binom{n}{k} \binom{a+k-1}{a-1} \frac{1}{a^k} \sim \frac{(a+k-1)!}{(a-1)! (k!)^2}
   \left( \frac{n}{a}\right)^k, \\
    \text{as} \quad  n
\rightarrow \infty \quad  \text{for fixed} \; a  \; \text{and} \;  k.
\end{multline}
(iii.) We have
 \begin{multline}\label{VarYkcountablecase}
\mathbb{E}[(Y_{n,k}^{\mathbf{(p)}})^2]=\sum\limits_{t=0}^{k} {n \choose 2k-t}\sum_{\substack{\{(a_1,\ldots,a_{t+1}) : a_i \in
\mathbb{N} \cup \{0\}\} \\ 
\{\lambda^1_i\}, \{\lambda^2_i\} \in S_{t+1, k-t}\quad}} \Big(\prod\limits_{i=1}^{t} p_{a_i}\Big) \prod\limits_{i=1}^{t+1} {\lambda^1_i+\lambda^2_i \choose \lambda^1_i} \\ \times \Big[\sum_{\substack{\{x^1_{ij}\} \in S_{a_i, \lambda^1_i} \\ 
 \{x^2_{ij}\} \in S_{a_i, \lambda^2_i}}} \prod\limits_{j=1}^{a_j} p_j^{x^1_{ij}x^2_{ij} } \Big].
\end{multline}
When $\mathbf{p}$ is the uniform distribution over a finite alphabet $[a]$, 
 \begin{multline}\label{eqn:2ndmomentYk} 
\mathbb{E}[(Y_{n,k}^{(a)})^2]  = \sum\limits_{t=0}^{k}  \sum\limits_{s=0}^{\min(a-1,k-t)} \Bigg[  a^{-(2k-t)} 4^{k-t}  {n \choose 2k-t}   \binom{k-t-s-1/2}{k-t-s} \\ \times \binom{s+(t+1)/2-1}{s}  
 \binom{2k-t-s+a-1}{t+2s} \binom{2k-2t-3s+a-1}{-s+a-1}\Bigg].
\end{multline} 
\end{theorem}

\begin{proof}
\textbf{(i.)} \eqref{EY_kgeneral} uses a standard combinatorial argument, so is skipped.

\noindent \textbf{(ii.}) The equation in \eqref{EYk}  follows immediately from the first part. The asymptotics for fixed  $a,k$ case is obtained by 
considering only the leading term (in terms of $n$).

To see that $\mathbb{E}[Y_{n,k}^{\mathbf{(p)}}]$ is maximized for the uniform
distribution, we will first show that function $$\psi(\mathbf{p}) = \mathbb{E}[Y_{n,k}^{\mathbf{(p)}}] = \binom{n}{k}
\sum_{(x_1,\ldots,x_a) \in S_{a,k}} \left(\Pi_{i=1}^a p_i^{x_i}\right)$$ is Schur-concave. To do so, we first observe that $\psi$ is symmetric, and so we can use Schur-Ostrowski criterion. That is, we just need to show that $(p_i - p_j ) \left(\frac{\partial \psi}{\partial p_i} - \frac{\partial \psi}{\partial p_j}\right) \leq 0$ for any $i \neq j$. We have 
\begin{align*}
(&p_i - p_j ) \left(\frac{\partial \psi}{\partial p_i} - \frac{\partial \psi}{\partial p_j}\right) \\=& (p_i - p_j ) \binom{n}{k} \sum_{(x_1,\ldots,x_a) \in S_{a,k}} \left(\left(\Pi_{s \neq i} p_s^{x_s} \right) x_i p_i^{x_i -1} - \left(\Pi_{s \neq j} p_s^{x_s} \right) x_j p_j^{x_j -1}\right)  
\\=&  (p_i - p_j ) \binom{n}{k} \sum_{(x_1,\ldots,x_a) \in S_{a,k}} \left(\left(\Pi_{s =1}^n p_s^{x_s} \right) \left(\frac{x_i}{p_i} - \frac{x_j}{p_j}  \right)\right) 
 \\=& (p_i - p_j ) \binom{n}{k} \sum_{(x_1,\ldots,x_a) \in S_{a,k}, x_i < x_j} \left(\left(\Pi_{s =1}^n p_s^{x_s} \right) \left\{\left(\frac{x_i}{p_i} - \frac{x_j}{p_j}  \right) + \left(\frac{x_j}{p_i} - \frac{x_i}{p_j}  \right) \right\}\right)  
\\=& (p_i - p_j ) \binom{n}{k} \sum_{(x_1,\ldots,x_a) \in S_{a,k}, x_i < x_j} \left(\left(\Pi_{s =1}^n p_s^{x_s} \right) \left\{\frac{(p_j - p_i)(x_i + x_j)}{p_i p_j} \right\}\right)  
\\=& -(p_i - p_j )^2 \binom{n}{k} \sum_{(x_1,\ldots,x_a) \in S_{a,k}, x_i < x_j} \left(\left(\Pi_{s =1}^n p_s^{x_s} \right) \left\{\frac{ x_i + x_j}{p_i p_j} \right\}\right)  \leq 0, 
\end{align*}
and so we are done.

\noindent  \textbf{(iii.)}  The proof below is for the uniform distribution over a finite alphabet $[a].$ The result stated in the theorem for the general case can be derived along the same lines.

For a given subset $I$  of $[n]$, let $\chi_I$ be the indicator function of $``\{X_i\}_{i \in I}$ is a weakly increasing subsequence". By this notation,
\begin{equation*}
\mathbb{E}[(Y_{n,k}^{(a)})^2]=\sum\limits_{\substack{{|I_1|=|I_2|=k}}} \mathbb{E}[\chi_{I_1} \times \chi_{I_2}].
\end{equation*}
 where the summation is over all pairs of subsequences of length $k$. 
 
Observe that $\chi_{I_1}$ and $\chi_{I_2}$ are not independent unless $I_1 \cap I_2 = \emptyset.$ So the idea of the first part of the proof is to write the sum over the partitions of $I_1$ and $I_2$, which are partitioned by their intersection, and over the partitions of the alphabet. So that conditioned on specified partitions, we obtain independent random variables in each part.

Let us introduce some notation for the proof. $S_{n,m}$ stands for the  set $\{(x_1,\ldots,x_n) : x_i \in
\mathbb{N} \cup \{0\}, x_1+\cdots+x_n = m\}$ as before. The size of the intersection,  $|I_1 \cap I_2|,$ is denoted by $t$ throughout the proof. So the subsequences excluding the intersection are partitioned into $t+1$ parts. The partitions are denoted by $\{ \lambda_i^1 \}, \{\lambda_i^2 \} \in S_{t+1, k-t}$, where $|I_1-I_1 \cap I_2|=|I_2-I_1 \cap I_2|=k-t$. As for the alphabet, we have $\{a_i\} \in S_{t+1, a-1}.$ In fact, any partition of the alphabet determines a weakly increasing subsequence indexed by $I_1 \cap I_2$. More specifically, if $I_1 \cap I_2 = \{ s_1, \ldots, s_t \}$, then take  $a_i=X_{s_i}-X_{s_{i-1}}$ for $2 \leq i \leq t$ and $a_1=X_{s_1}-1,$ $a_{t+1}=a-X_{s_{t}}.$

First, we write the sum over the size of the intersection as
\begin{equation*}
\mathbb{E}[(Y_{n,k}^{(a)})^2]=\sum\limits_{t=0}^{k}  \sum_{\substack{{|I_1|=|I_2|=k} \\ 
|I_1 \cap I_2|=t}} \mathbb{E}[\chi_{I_1} \times \chi_{I_2}]. 
\end{equation*} 
Then the inner sum, first choosing the position of $I_1 \cup I_2$ in the sequence, is written over the partitions defined above. We count the number of weakly increasing subsequences $I_1$ and $I_2$ confined to a given partition triple, noting that each subsequence pair has probability $a^{-(2k-t)}$ to occur.
\begin{flalign*}
&\sum_{\substack{{|I_1|=|I_2|=k} \\ 
|I_1 \cap I_2|=t}} \mathbb{E}[\chi_{I_1} \times \chi_{I_2}]= && \\ & {n \choose 2k-t} \sum_{\substack{\{a_i\} \in S_{t+1, a-1} \\ 
\{\lambda^1_i\}, \{\lambda^2_i\} \in S_{t+1, k-t}\quad}} a^{-(2k-t)}  \prod\limits_{i=1}^{t+1} {\lambda^1_i+\lambda^2_i \choose \lambda^1_i} {a_i+\lambda^1_i \choose a_i }{a_i+\lambda^2_i\choose a_i}.&&
\end{flalign*}

Before giving a precise explanation for the summand, it might be worth presenting an example on the partition indices of the sum. The partitions we are summing over can be described graphically as in the Example \ref{example} below. Associate each $\{\lambda^1_i, \lambda^2_i,a_i\}$ triple either to a rectangular block (if $a_i \geq 1$) or to a line (if $a_i=0$), then attach them diagonally. The $i^{th}$ part has width $\lambda^1_{i}+\lambda^2_{i}$ and height $a_i$. The sum runs through all such arrangements lying in the rectangle of size $(2k-t) \times (a-1).$

\begin{example} \label{example}

\hspace*{0.5cm}

\hspace*{-0.6cm}
\begin{tikzpicture}[scale=0.90]

\draw[step=0.5cm,color=gray] (-1,-1) grid (11.5,3);
\draw [thick, <->] (-1,3.2) -- (-1,-1) -- (11.7,-1);

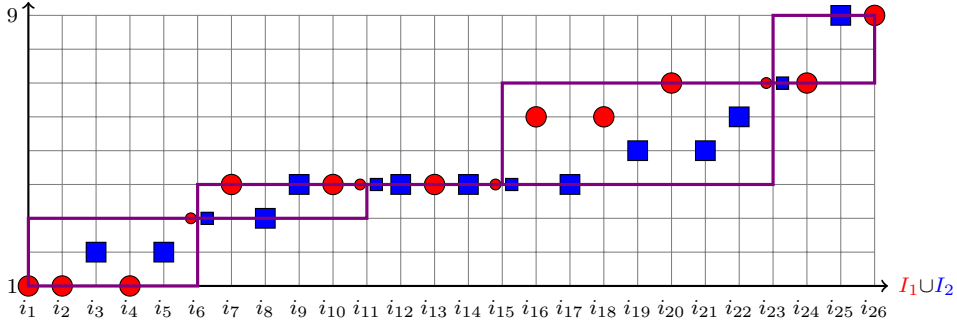
\draw [fill=red](-1,-1) circle [radius=0.15];
\draw [fill=red](-0.5,-1) circle [radius=0.15];
\node[draw, rectangle, fill=blue] at (0,-0.5) {};
\draw [fill=red](0.5,-1) circle [radius=0.15];
\node[draw, rectangle, fill=blue] at (1,-0.5) {};
\draw [fill=red](1.4,0) circle [radius=0.08];
\draw [fill=blue](1.55,-0.09) rectangle (1.73,0.09);
\draw [fill=red](2,0.5) circle [radius=0.15];
\node[draw, rectangle, fill=blue] at (2.5,0) {};
\node[draw, rectangle, fill=blue] at (3,0.5) {};
\draw [fill=red](3.5,0.5) circle [radius=0.15];
\draw [fill=red](3.9,0.5) circle [radius=0.08];
\draw [fill=blue](4.05,0.41) rectangle (4.23,0.59);
\node[draw, rectangle, fill=blue] at (4.5,0.5) {};
\draw [fill=red](5,0.5) circle [radius=0.15];
\node[draw, rectangle, fill=blue] at (5.5,0.5) {};
\draw [fill=red](5.9,0.5) circle [radius=0.08];
\draw [fill=blue](6.05,0.41) rectangle (6.23,0.59);
\draw [fill=red](6.5,1.5) circle [radius=0.15];
\node[draw, rectangle, fill=blue] at (7,0.5) {};
\draw [fill=red](7.5,1.5) circle [radius=0.15];
\node[draw, rectangle, fill=blue] at (8,1) {};
\draw [fill=red](8.5,2) circle [radius=0.15];
\node[draw, rectangle, fill=blue] at (9,1) {};
\node[draw, rectangle, fill=blue] at (9.5,1.5) {};
\draw [fill=red](9.9,2) circle [radius=0.08];
\draw [fill=blue](10.05,1.91) rectangle (10.23,2.09);
\draw [fill=red](10.5,2) circle [radius=0.15];
\node[draw, rectangle, fill=blue] at (11,3) {};
\draw [fill=red](11.5,3) circle [radius=0.15];
\draw [violet, very thick] (-1,-1) rectangle (1.5,0);
\draw [violet, very thick] (1.5,0) rectangle (4,0.5);
\draw [violet, very thick] (4,0.5) rectangle (6,0.5);
\draw [violet, very thick] (6,0.5) rectangle (10,2);
\draw [violet, very thick] (10,2) rectangle (11.5,3);

\node [above] at (-1,3.2) {\small alphabet};
\foreach \x in {1,,9}
\node[left] at (-1,.5*\x-1.5) {$\scriptstyle{\x}$ };
\node [right] at (11.7,-1) {$\scriptstyle{{\color{red}{I_1}} \cup {\color{blue}{I_2}}}$};
\foreach \x in {1,...,26}
\node[below] at (.5*\x-1.5,-1.05) {$\scriptstyle{i_{\x}}$ };
\end{tikzpicture}

\centerline{$k=15, t=4, a=9, \quad \quad {\color{red}{I_1}} \cap {\color{blue}{I_2}} =\{i_6, i_{11}, i_{16}, i_{24}\}$,}  
$\{a_i\}_{i=1}^{t+1}=(2,1,0,3,2), \, \{\lambda^1_i\}_{i=1}^{t+1}=(3,2,1,3,2), \, \{\lambda^2_i\}_{i=1}^{t+1}=(2,2,2,4,1)$.

\bigskip

The disks and the squares represent two different subsequences. An intersection point is denoted by a small disk and small square juxtaposed. The horizontal axis shows the indices in the sequence (for simplicity the size of the sequence is the size of the union of the two subsequences) and the vertical axis is for the values from the alphabet.

The figure represents one possible allocation in the sum over the partitions above. The union of two subsequences is partitioned by their intersection, and each part corresponds to a $\{\lambda_i^1, \lambda_i^2, a_i\}$ index of the product in the formula above.  
\end{example}

There are ${\lambda^1_i+\lambda^2_i \choose \lambda^1_i}$ positions for the elements of $I_1$ and $I_2$ that are in the $i^{th}$ part. Then, for a given $a_i$, we count how many possibilities there are to have two increasing subsequences within the same part. Because the distribution over the alphabet is uniform, the probability of each of them is $a^{-(2k-t)}.$ Finally, we have

\begin{flalign}\label{sum over part}
&\mathbb{E}[(Y_{n,k}^{(a)})^2] =&& \notag\\ &\sum\limits_{t=0}^{k} {n \choose 2k-t} a^{-(2k-t)} \mkern-20mu \sum_{\substack{\{a_i\} \in S_{t+1, a-1} \\ 
\{\lambda^1_i\}, \{\lambda^2_i\} \in S_{t+1, k-t}\quad}} \prod\limits_{i=1}^{t+1} {\lambda^1_i+\lambda^2_i \choose \lambda^1_i} {a_i+\lambda^1_i \choose a_i }{a_i+\lambda^2_i\choose a_i}.&& \notag \\
\end{flalign} 

The equation for the general case, where the distribution is arbitrary over the alphabet (See Theorem \ref{thm:totalnumberofWIS}), can be obtained similarly. 

Now in order to evaluate the sum that runs through the partitions of $k-t$ and $a-1$ in equation \ref{sum over part} above, we consider it as the coefficient of term $x^{k-t}y^{k-t}z^{a-1}$ for the generating function of
\begin{equation*}
 \prod\limits_{i=1}^{t+1} \sum\limits_{\lambda^1_i,\lambda^2_i,a_i \geq 0} {\lambda^1_i+\lambda^2_i \choose \lambda^1_i} {a_i+\lambda^1_i \choose a_i }{a_i+\lambda^2_i\choose a_i}x^{\lambda^1_i} y^{\lambda^2_i} z^{a_i}.
\end{equation*}
The notation $[x^n]$ attached before a sum denotes \textit{the coefficient of $x^n$ in the sum}. The following proposition, which is proved in the Appendix \ref{pfgen1}, concludes the proof. 
\qed
\begin{proposition} \label{propgen1} 

\begin{gather*}
[x^{k-t}y^{k-t}z^{a-1}]\prod\limits_{i=1}^{t+1} \sum\limits_{\lambda^1_i,\lambda^2_i,a_i \geq 0} {\lambda^1_i+\lambda^2_i \choose \lambda^1_i} {a_i+\lambda^1_i \choose a_i }{a_i+\lambda^2_i\choose a_i}x^{\lambda^1_i} y^{\lambda^2_i} z^{a_i} =\\
 4^{k-t} \sum\limits_{s=0}^{\min(a-1,k-t)} \Bigg[ 
 \binom{k-t-s-1/2}{k-t-s} \binom{s+(t+1)/2-1}{s} \\
\times  \binom{2k-t-s+a-1}{t+2s} \binom{2k-2t-3s+a-1}{-s+a-1}\Bigg].
\end{gather*}

\end{proposition}

The following theorem summarizes our results on the total number of weakly increasing
subsequences in a random word.

\begin{theorem}\label{thm:totalnumberofWIS}
(i.) We have $$\mathbb{E}[Y_{n}^{(\mathbf{p})}] = \sum_{k=0}^n \binom{n}{k}
\sum_{(x_1,\ldots,x_a) \in S_{n,k}} \left(\Pi_{i=1}^a
p_i^{x_i}\right),$$ where $S_{n,k} = \{(x_1,\ldots,x_n) : x_i \in
\mathbb{N} \cup \{0\}, x_1+\cdots+x_a = k\}$. This expression is
maximized for the uniform distribution over $[a]$ for which we have
$\mathbb{E}[Y_n^{(a)}] = \sum_{k=0}^n \binom{n}{k} \binom{a+k-1}{a-1}
\frac{1}{a^k}$. When $a=2$, this simplifies to
$$\mathbb{E}[Y_n^{(2)}] = \left(\frac{3}{2}\right)^{n-1} \left(\frac{n+3}{2}\right).$$
For fixed $a \geq 2$, one has the asymptotic
$$\mathbb{E}[Y_n^{(a)}] \sim \frac{1}{(a-1)! a^{a-1}} n^{a-1} \left(1+\frac{1}{a} \right)^{n-a+1}, \quad \text{as} \quad n \rightarrow \infty.$$

\noindent (ii.) For the uniform case, we have the following upper and lower
bounds for the second moment:
$$\mathbb{E}[(Y_n^{(a)})^2] \leq \frac{n!}{(a-1)!}\sum_{0 \leq k_1 + k_2 \leq n} 2^{k_1} \frac{(a+k_1-1)!(a+k_2-1)!}{(k_1!)^2(k_2!)^2 (n-k_1-k_2)! a^{k_1+k_2}}$$
and
$$\mathbb{E}[(Y_n^{(a)})^2] \geq \frac{n!}{(a-1)!}\sum_{0 \leq k_1 + k_2 \leq n} \frac{(a+k_1-1)!(a+k_2-1)!}{(k_1!)^2(k_2!)^2 (n-k_1-k_2)! a^{k_1+k_2}}.$$
%
%
When $a=2$, these bounds simplify to
\begin{gather*}
\sum_{0 \leq k_1 + k_2 \leq n} \frac{k_2+1}{2^{k_1 + k_2}}  \binom{n}{k_1,k_2,n-k_1-k_2} \\ \leq \mathbb{E}[(Y_n^{(2)})^2]
\leq \sum_{0 \leq k_1 + k_2 \leq n} \frac{k_2+1}{2^{k_2}} \binom{n}{k_1,k_2,n-k_1-k_2}
\end{gather*}

\end{theorem}

\begin{remark}Note that the upper and lower bounds for $a = 2$ case differ only by the $2^{k_1}$ 
term.\end{remark}

\begin{proof} \textbf{(i.)} The first
expression follows from the expectation of $Y_{n,k}^{\mathbf{(p)}}$ in Theorem
\ref{thm:Ykmoment}.

When $a=2$, we have $$\mathbb{E}[Y_{n}^{(2)}] = \sum_{k=0}^n
\binom{n}{k} (k+1) \frac{1}{2^k} = \sum_{k=1}^n \binom{n}{k}
\frac{k}{2^k} + \sum_{k=0}^n \binom{n}{k} \frac{1}{2^k}.$$ The
formula for this special case now follows from some elementary
manipulations and by using the binomial theorem.

For the asymptotics, first note that $$\mathbb{E}[Y_n^{(a)}] =
\sum_{k=0}^n \binom{n}{k} \frac{(k+a-1)!}{k! (a-1)!}
\frac{1}{a^k}.$$ So,
\begin{align*}
  \mathbb{E}[Y_n^{(a)}] &\sim \frac{1}{(a-1)!} \sum_{k=0}^n \frac{n!}{k! (n-k)!} k^a \frac{1}{a^k} \\
    &\sim \frac{1}{(a-1)!} \sum_{k=a-1}^n \frac{n!}{(k-a+1)! (n-k)!} \frac{1}{a^k}  \\
    &= n(n-1)\ldots(n-a+2) \frac{1}{(a-1)!} \sum_{k=a-1}^n \frac{(n-a+1)!}{(k-a+1)! (n-k)!}
    \frac{1}{a^k} \\
    &\sim \frac{n^{a-1}}{(a-1)!} \sum_{k=0}^{n-a+1}
    \binom{n-a+1}{k} \frac{1}{a^{k+a-1}} \\
    &\sim \frac{n^{a-1}}{(a-1)!} \frac{1}{a^{a-1}} \left(1+\frac{1}{a}
    \right)^{n-a+1},
\end{align*}
where we used the fact that $a$ is fixed,  and also the binomial
theorem for the last step.

\noindent \textbf{(ii.)} Let us begin with the upper bound. First note that
\begin{align*}
   \mathbb{E}[(Y_n^{(a)})^2]  =& \mathbb{E}\left[\sum_{j,k=0}^n  Y_{n,j}^{(a)} Y_{n,k}^{(a)} \right] \\
  \begin{split}
  =& \mathbb{E}\Big[\sum_{j,k=0}^n \sum_{1 \leq i_1 < \cdots < i_j \leq n}  \sum_{1 \leq r_1 < \cdots < r_k \leq n}
   \Big(\mathbf{1}(X_{i_1} \leq \cdots \leq X_{i_j})\\
    &\times \mathbf{1}(X_{r_1} \leq \cdots \leq X_{r_k}) \Big)\Big]
  \end{split}\\
   &= \sum_{S_1,S_2} \mathbb{E}[\chi_{S_1} \chi_{S_2}]
\end{align*}
where the summation $\Sigma_{S_1,S_2}$ is over all subsets $S_1,S_2$
of $[n]$ and where $\chi_{S_1}$ and $\chi_{S_2}$ are the indicators
that $X_1,\ldots,X_n$ reduced to $S_1$ and $S_2$ are weakly increasing,
respectively.
So we have
\begin{equation*}
 \mathbb{E}[(Y_n^{(a)})^2] = \sum_{S_1,S_2} \mathbb{E}[\chi_{S_1} \chi_{S_2}]
   \leq \sum_{S_1,S_2} \mathbb{E}[\chi_{S_1} \chi_{S_2 - S_1}]
    = \sum_{S_1,S_2} \mathbb{E}[\chi_{S_1}] \mathbb{E}[\chi_{S_2 -
    S_1}],
\end{equation*}
where we used the independence of $(X_i)_{i \in S_1}$ and $(X_j)_{j
\in S_2 - S_1}$ for the last step. Thus,
$$\mathbb{E}[(Y_n^{(a)})^2] \leq \sum_{k_1 + k_2 \leq n} \sum_{(S_1,S_2)}\binom{a+k_1-1}{a-1}\binom{a+k_2-1}{a-1} \frac{1}{a^{k_1+k_2}},$$
where the inner summation is taken over ordered sets of subsets
$S_1, S_2$ such that $|S_1| = k_1$ and $|S_2 - S_1| = k_2$.

Noting that the last observation is equivalent to
$$\mathbb{E}[(Y_n^{(a)})^2] \leq\sum_{k_1 + k_2 \leq n} \binom{a+k_1-1}{a-1}\binom{a+k_2-1}{a-1} \frac{1}{a^{k_1+k_2}} \sum_{(S_1,S_2)}
1,$$ we should next estimate $\sum_{(S_1,S_2)} 1$. We have
\begin{equation*}
  \sum_{(S_1,S_2)} 1 \leq \binom{n}{k_1}   2^{k_1} \binom{n-k_1}{k_2}.
\end{equation*}
Hence, combining these observations we arrive at
\begin{align*}
 \mathbb{E}[(Y_n^{(a)})^2] &\leq& \sum_{\mathclap{0 \leq k_1  + k_2 \leq n}} 2^{k_1} \binom{a+k_1-1}{a-1}\binom{a+k_2-1}{a-1} \binom{n}{k_1} \binom{n-k_1}{k_2} \frac{1}{a^{k_1+k_2}} \\
   &=& \frac{n!}{((a-1)!)^2}\sum_{0 \leq k_1 + k_2 \leq n} 2^{k_1} \frac{(a+k_1-1)!(a+k_2-1)!}{(k_1!)^2(k_2!)^2 (n-k_1-k_2)!
   a^{k_1+k_2}},
\end{align*}
where the last step requires some elementary manipulations. This
concludes the proof of the upper bound.

For the binary case, the bound we derived simplifies as
\begin{align*}
  \mathbb{E}[(Y_n^2)^{(2)}]  \leq& \sum_{0 \leq k_1 + k_2 \leq n} \frac{(k_1+1)!(k_2+1)!}{k_1! k_2!} \frac{2^{k_1}}{2^{k_1+k_2}} \binom{n}{k_1} \binom{n-k_1}{k_2}\\
    =& \sum_{0 \leq k_1 + k_2 \leq n} (k_1+1)(k_2+1) \frac{1}{2^{k_2}} \binom{n}{k_1} \binom{n-k_1}{k_2}  \\
   =& \sum_{k_2 = 0}^{n} \frac{k_2+1}{2^{k_2}} \sum_{k_1=0}^{n-k_2} (k_1+1)
   \binom{n}{k_1,k_2,n-k_1-k_2}.
\end{align*}

Next, let us focus on the lower bound. First observe that
\begin{equation*}
  \mathbb{E}[(Y_n^{(a)})^2] \geq \sum_{\overline{(S_1,S_2)}} \mathbb{E}[\chi_{S_1} \chi_{S_2}] = \sum_{\overline{(S_1,S_2)}} \mathbb{E}[\chi_{S_1}]
    \mathbb{E}[\chi_{S_2}],
\end{equation*}
where $\overline{(S_1,S_2)}$ is so that $S_1,S_2 \subset [n]$ and
$S_1 \cap S_2 = \emptyset$. Then
\begin{align*}
  \mathbb{E}[(Y_n^{(a)})^2] \geq& \sum_{k_2 = 0}^n \sum_{k_1 =0}^{n-k_2}
  \binom{a + k_2 -1}{a-1} \binom{a + k_1 -1}{a-1} \binom{n}{k_2}
  \binom{n-k_2}{k_1} \frac{1}{a^{k_1+k_2}} \\
  =& \frac{n!}{((a-1)!)^2}\sum_{0 \leq k_1 + k_2 \leq n} \frac{(a+k_1-1)!(a+k_2-1)!}{(k_1!)^2(k_2!)^2 (n-k_1-k_2)!
  a^{k_1+k_2}}.
\end{align*}
This simplifies to
$$ \mathbb{E}[(Y_n^{(2)})^2] \geq \sum_{0 \leq k_1 + k_2 \leq n} \frac{k_2+1}{2^{k_1 + k_2}}  \binom{n}{k_1,k_2,n-k_1-k_2}
$$ for the binary case.  \hfill 
\end{proof}

\section{Central limit theorem for random words}\label{sec:CLTwords}

\begin{theorem}Let $d_K$ denote the Kolmogorov distance, and $\mathcal{G}$ denote the standard Gaussian random variable. \label{thm:CLTwords}
 We have
$$d_K\left(\frac{Y_{n,k}^{\mathbf{(p)}} - \mathbb{E}[Y_{n,k}^{\mathbf{(p)}}]}{\beta_{n}},
\mathcal{G} \right)  \leq \frac{C}{\sqrt{n}},$$ where $\beta_n \sim
\sqrt{Var(Y_{n,k}^{\mathbf{(p)}})}$ as $n \rightarrow \infty$, $\mathbb{E}[Y_{n,k}^{\mathbf{(p)}}]$ is given by \eqref{EYk} and $Var(Y_{n,k}^{\mathbf{(p)}}) = \mathbb{E} [(Y_{n,k}^{\mathbf{(p)}})^2] - (\mathbb{E}[Y_{n,k}^{\mathbf{(p)}}])^2$ with $\mathbb{E} [(Y_{n,k}^{\mathbf{(p)}})^2]$ as in \eqref{VarYkcountablecase}. $C$ is constant with respect to $n$, but depends on $k$ and $\mathbf{p}.$ \\
In particular, $(Y_{n,k}^{\mathbf{(p)}} -
\mathbb{E}[Y_{n,k}^{\mathbf{(p)}}])/\sqrt{Var(Y_{n,k}^{\mathbf{(p)}})}$ converges in
distribution to $\mathcal{G}$.
\end{theorem}

Before giving the proof, let us  discuss a special case of Theorem
\ref{thm:CLTwords}. To begin with, for a given sequence of real numbers
$\mathbf{x}=(x_1,\ldots,x_n)$, the \emph{number of inversions} in
$\mathbf{x}$ is defined by
$$Inv(\mathbf{x})=\#\{(i,j) : 1 \leq i \leq j \leq n, x_i> x_j\}.$$
Number of inversions (and other related descent statistics) is a standard tool
in nonparametric statistics to check the randomness of a given word
or permutation. Using Theorem \ref{thm:CLTwords} with $k=2$
(and simplifying the expectation and variance formulas) reveals the
asymptotic normality of the number of inversions in a random word as
a corollary. This was previously studied  by Bliem/Kousidis
\cite{bliem} and Janson \cite{Janson}.

\begin{corollary}\label{cor:inversionsinwords} (\cite{bliem}, \cite{Janson})
Let $X=(X_1,\ldots,X_n)$ be a random vector where $X_1,\ldots,X_n$ are
independent random variables that are uniformly distributed over
$[a]$. Then we have
$$\frac{Inv(X) - \frac{a-1}{a} \frac{n(n-1)}{4}}{\sqrt{\frac{a^2-1}{a^2} \frac{n(n-1)(2n+5)}{72}}} \longrightarrow_d \mathcal{G}, \qquad \text{as} \; \; n\rightarrow \infty.$$ 
\end{corollary}
Corollary \ref{cor:inversionsinwords} provides a central limit
theorem for the number of inversions  in riffle shuffles as well. This will be further explored in Section
\ref{sec:riffle}. Moreover, the technique used in proof of Theorem \ref{thm:CLTwords} will also help us to  show that
$Z_{n,k}$ satisfies a central limit theorem after the `natural'
centering and scaling (in contrast with $Z_{n}$). We were
not able to find a proof for this result in literature except for the special case $k=2$. See \cite{fulman2} and \cite{pike}.

\vspace{0.1in}

The proof of Theorem \ref{thm:CLTwords} will make use of a result of \cite{chen} on asymptotic normality of $U-$statistics. We provide some necessary background on U-statistics. First,  for a real valued symmetric function $g : \mathbb{R}^m \rightarrow \mathbb{R}$ and for a random sample
$X_1,\ldots,X_n$ with $n \geq m$, \emph{a U-statistic with kernel} $g$
is defined by $$U_n=U_n(g) = \frac{1}{\binom{n}{m}} \sum_{C_{m,n}}
g(X_{i_1},\ldots,X_{i_m})$$ where the summation is over the set
$C_{m,n}$ of all $\binom{n}{m}$ combinations of $m$ integers, $i_1 <
i_2<\cdots<i_m$ chosen from $\{1,\ldots,n\}.$  We also set 
$g_1(X_1):=\mathbb{E}[g(X_1,\ldots,X_m)|X_1]$.

\begin{theorem}\label{conv} \cite{chen}
Let $X_1,...,X_n$ be i.i.d. random variables, $U_n$ be a U-statistic
with symmetric kernel $g$, $\mathbb{E}[g(X_1,...,X_m)]=0,
\sigma^2=Var(g(X_1,...,X_m))<\infty$ and
$\sigma_1^2=Var(g_1(X_1))>0.$  If in addition
$\mathbb{E}|g_1(X_1)|^3< \infty,$ then $$d_K \left(\frac{\sqrt{n}}{m
\sigma_1}U_n,Z \right)  \leq \frac{6.1
\mathbb{E}|g_1(X_1)|^3}{\sqrt{n} \sigma_1^3}
+\frac{(1+\sqrt{2})(m-1)\sigma}{(m(n-m+1))^{1/2} \sigma_1}.
$$
\end{theorem}

\noindent \textbf{Proof of Theorem \ref{thm:CLTwords}.} Let
$X_1,\ldots,X_n$ be i.i.d with distribution $\mathbb{P}(X_i = j) =
p_j$, $j \in [a]$. Also let $U_1,\ldots,U_n$ be independent (also
independent of $X_i$'s) random variables uniform over $(0,1)$. Let
$\sigma$ be a random permutation in $S_n$ so that
\begin{equation*}
    U_{\sigma(1)}<\cdots< U_{\sigma(n)}.
\end{equation*}
Then, clearly, $$(X_1,\ldots,X_n)=_d
(X_{\sigma(1)},\ldots,X_{\sigma(n)}).$$ For $1 \leq k \leq n$, set
$$\mathcal{S}_1=\{(i_1,\ldots,i_k): i_j \in [n], j=1,\ldots,k \; \text{and} \; 1\leq i_1<\cdots<i_k \leq n\},$$ and
$$\mathcal{S}_2=\{(i_1,\ldots,i_k): i_j \in [n], j=1,\ldots,k \; \text{and} \; i_1,\ldots,i_k \; \text{are distinct}\}.$$
Next, observe that
\begin{align*}
  Y_{n,k}^{\mathbf{(p)}} =& \sum_{(i_1,\ldots,i_k) \in \mathcal{S}_1} \mathbf{1}(X_{i_1} \leq X_{i_2} \leq \ldots \leq X_{i_k})\\
    =_d& \sum_{(i_1,\ldots,i_k) \in \mathcal{S}_1} \mathbf{1}(X_{\sigma(i_1)} \leq X_{\sigma(i_2)} \leq \ldots \leq X_{\sigma(i_k)}) \\
    =& \sum_{(i_1,\ldots,i_k) \in \mathcal{S}_2} \mathbf{1}(X_{\sigma(i_1)} \leq X_{\sigma(i_2)} \leq \ldots \leq X_{\sigma(i_k)}, i_1 < \ldots <i_k) \\
    =& \sum_{(i_1,\ldots,i_k) \in \mathcal{S}_2} \mathbf{1}(X_{\sigma(i_1)} \leq X_{\sigma(i_2)} \leq \ldots \leq X_{\sigma(i_k)}, U_{\sigma(i_1)} < \ldots
    <U_{\sigma(i_k)}) \\
    =_d& \sum_{(i_1,\ldots,i_k) \in \mathcal{S}_2} \mathbf{1} (X_{i_1} \leq X_{i_2} \leq \ldots \leq X_{i_k}, U_{i_1} < \ldots <
    U_{i_k}).
\end{align*}
Now, define functions $f$ and $g$ by setting 
$$\binom{n}{2}f((x_{i_1},u_{i_1}),(x_{i_2},u_{i_2}),\ldots,(x_{i_k},u_{i_k}))=\mathbf{1}(x_{i_1} \leq \ldots \leq x_{i_k}, u_{i_1} < \ldots < u_{i_k}),$$  and
\begin{gather*}g((x_{i_1},u_{i_1}),(x_{i_2},u_{i_2}),\ldots,(x_{i_k},u_{i_k})) \\ = \sum_{(j_1,\ldots,j_k) \in \mathcal{S}_{i_1,\ldots,i_k}}f((x_{j_1},u_{j_1}),(x_{j_2},u_{j_2}),\ldots,(x_{j_k},u_{j_k})),
\end{gather*}
where $S_{i_1,\ldots,i_k}$ is the set of all permutations of
$i_1,\ldots,i_k$. Then, we can express $Y_{n,k}^{\mathbf{(p)}}$in terms of $g$
as
$$Y_{n,k}^{\mathbf{(p)}}=\frac{1}{\binom{n}{k}}\sum_{(i_1,\ldots,i_k) \in \mathcal{S}_1}g((X_{i_1},U_{i_1}),(X_{i_2},U_{i_2}),\ldots,(X_{i_k},U_{i_k})).$$
This reveals that $Y_{n,k}^{\mathbf{(p)}}$ is indeed a $U$-statistics since
(i.) $g$ is symmetric,
(ii.) $g$ is a function of random vectors whose coordinates are
  independent,
(iii.) $g \in L^2$.
Result now follows by noting the standard asymptotic formula in
$U$-statistics theory \cite{Lee}, $$\frac{m \sigma_1}{\sqrt{n}} \sim
\sqrt{Var(Y_{n,k}^{\mathbf{(p)}})},$$ and by Theorem \ref{conv}. 
\end{proof}

\section{Moments and a central limit theorem for random permutations}\label{momentsCLTperms}

\begin{theorem}\label{thm:CLTZnk}
We have
$$\frac{Z_{n,k} - \mathbb{E}[Z_{n,k}] }{\sqrt{Var(Z_{n,k})}}
\longrightarrow_d \mathcal{G}, \qquad \text{as} \; \; n\rightarrow \infty,$$  where
$$\mathbb{E}[Z_{n,k}]= \binom{n}{k} \frac{1}{k!},$$ and $Var(Z_{n,k}) = \mathbb{E}[Z_{n,k}^2] - (\mathbb{E}[Z_{n,k}])^2$, as
\begin{equation}
\begin{gathered}\label{eqn:2ndmomentZk}
\mathbb{E}[Z_{n,k}^2] = \sum_{t+s \leq k} \Bigg[ [(2 k - t)!]^{-1} 4^{k-t} \binom{n}{2k-t}\binom{k-t-s-1/2}{k-t-s}\\ \times \binom{s+(t+1)/2-1}{s} \binom{2k-t}{2k-2t-2s}\Bigg].
\end{gathered}
\end{equation}
\end{theorem}

\begin{remark}\label{rmk:Zk} For the special case $k=2$, the proof given here provides an
alternative for the approaches of \cite{fulman2} and \cite{pike} on the
number of inversions in uniformly random permutations. However, the
technique  of the cited papers (exchangeable pairs of Stein's method)
are in a certain sense more general and they also apply to
generalized descents of random permutations. 
\end{remark}

\noindent \textbf{Proof of Theorem \ref{thm:CLTZnk}:} The central limit theorem is based on the following simple result which is attributed to R\'enyi: 
 If $\mathbf{U}=(U_1,\ldots,U_n)$ is a random
vector where $U_i$'s are independent $U(0,1)$ random variables, and if  $R_i$ is  rank 
of $U_i$ in $(U_1,\ldots,U_n)$, then $(R_1,\ldots,R_n)$ has the same distribution with a 
uniformly random permutation in $S_n$. So, leaving the computation of first two moments aside for now, 
$$Z_{n,k} =_d \sum_{1 \leq i_1 < \cdots < i_k \leq n} \mathbf{1} (U_{i_1} < \cdots <
U_{i_k}),$$ where $U_1,\ldots, U_n$ are independent random variables
that are uniformly distributed over $(0,1)$, and the result  follows
by following the same steps in the proof of Theorem
\ref{thm:CLTwords} and by a straightforward application of
Slutsky's theorem. 

\vspace{0.15in}

Next, we derive the first two moments of $Z_{n,k}$. The computation for the first moment is straightforward. For the second moment, we follow the same approach we used in derivation of the second moment in  Theorem   \ref{thm:Ykmoment}. 
We define $\chi_I$ to be the indicator function that the subsequence indexed by   $I \subseteq [n]$ is increasing. Therefore,
\begin{equation*}
\mathbb{E}[Z_{n,k}^2]=\sum\limits_{\substack{{|I_1|=|I_2|=k}}} \mathbb{E}[\chi_{I_1} \times \chi_{I_2}].
\end{equation*}
 where the summation is over all pairs of subsequences of length $k$.  
The idea is the same with the random word case. We consider the sum over the partitions of $I_1$ and $I_2$, which are partitioned by their intersection. So let $t$ denote the size of $I_1 \cap I_2,$ and $\{ \lambda_i^1 \}, \{\lambda_i^2 \} \in S_{t+1, k-t}$ be partitions of $k-t,$ where $|I_1-I_1 \cap I_2|=|I_2-I_1 \cap I_2|=k-t.$  So that conditioned on specified partitions, we obtain independent random variables in each part as before.

First write the sum as 

\begin{equation*}
\mathbb{E}[Z_{n,k}^2] = \sum\limits_{t=0}^{k}  {n \choose 2k-t}\sum_{\substack{{|I_1|=|I_2|=k} \\ 
|I_1 \cap I_2|=t}}\mathbb{E}[\chi_{I_1} \times \chi_{I_2}] \\
\end{equation*}

Next we count the number of increasing subsequences $I_1$ and $I_2$ for given partitions. For each $(\lambda^1_i, \lambda^2_i)$ pair, there are ${\lambda^1_i+\lambda^2_i \choose \lambda^1_i}$ positions for the elements of $I_1$ and $I_2$ in the $i^{th}$ part. Then we choose $\lambda^1_i+\lambda^2_i$ between the random variables indexed by $I_1$ and the random variables indexed by $I_2.$ So we have one more factor of ${\lambda^1_i+\lambda^2_i \choose \lambda^1_i}.$ Being chosen that way, there is one possibility to put them in increasing order. Finally, we note that the probability of any permutation of $I_1 \cup I_2$ is $ [(2 k - t)!]^{-1}.$ Therefore, 
\begin{equation}\label{sum over part2}
\mathbb{E}[Z_{n,k}^2]=\sum\limits_{t=0}^{k} {n \choose 2k-t}  [(2 k - t)!]^{-1}  \sum_{\substack{
\{\lambda^1_i\}, \{\lambda^2_i\} \in S_{t+1, k-t}\quad}} \prod\limits_{i=1}^{t+1} {\lambda^1_i+\lambda^2_i \choose \lambda^1_i}^2.
\end{equation}

Now we evaluate the sum over the partitions of $k-t$ in  \eqref{sum over part2} as in the proof of Theorem \ref{thm:Ykmoment}. We consider it as the coefficient of  $x^{k-t}y^{k-t}$ for the generating function of
\begin{equation*}
 \prod\limits_{i=1}^{t+1} \sum\limits_{\lambda^1_i,\lambda^2_i \geq 0} {\lambda^1_i+\lambda^2_i \choose \lambda^1_i}^2 
\end{equation*}
The following proposition is proved in the Appendix \ref{pfgen2}, from which the result follows.  \qed
\begin{proposition} \label{propgen2}

\begin{flalign*}
&[x^{k-t}y^{k-t}]\prod\limits_{i=1}^{t+1} \sum\limits_{\lambda^1_i,\lambda^2_i\geq 0} {\lambda^1_i+\lambda^2_i \choose \lambda^1_i}^2 x^{\lambda^1_i} y^{\lambda^2_i} =&&\\ &
4^{k-t}\sum\limits_{s=0}^{k-t} \Bigg[  \binom{k-t-s-1/2}{k-t-s} \binom{s+(t+1)/2-1}{s} \binom{2k-t}{2k-2t-2s}\Bigg].&&
\end{flalign*}

\end{proposition}

\section{Second moment asymptotics in random permutations}\label{sec:secondmomentinperms}
The purpose of this section is to study the asymptotic behavior of  $\mathbb{E}[Z_{n,k}^2]$ as $n \rightarrow \infty$.

\begin{theorem} (The asymptotics for the second moment and the variance of $Z_{n,k}$) \label{asZk}\\ \\
(i.) We have
 \begin{equation*}
\mathbb{E}[Z_{n,k}^2]\sim \frac{n^{2k}}{(k!)^4}, \qquad n \rightarrow \infty,
\end{equation*}
where $\mathbb{E}[Z_{n,k}^2]$ is given in \eqref{eqn:2ndmomentZk}. \\
(ii.)  We have
 \begin{equation}\label{VarZk}
\text{Var}(Z_{n,k})\sim \frac{1}{2((2k-1)!)^2} \left[ \binom{4k-2}{2k-1}-2{2k-1 \choose k}^2 \right] n^{2k-1}, \qquad n \rightarrow \infty,
\end{equation}
where the first two moments of $Z_{n,k}$ are given in Theorem \ref{thm:CLTZnk}.
\end{theorem}

\begin{remark}
Observe that in the expansion 
\begin{equation*}
\binom{4k-2}{2k-1} = \sum_{i=0}^{2k-1} \binom{2k-1}{i}^2,
\end{equation*}
the term we subtract in \eqref{VarZk} is the sum of middle terms. So there might an interesting interpretation or an easier proof for the asymptotics of the variance.
\end{remark}

\begin{proof}
Observe that the largest order term is $n^{2k}$ for both $\mathbb{E}[(Z_{n,k})^2]$ and $\mathbb{E}[(Z_{n,k})]^2.$ Express
\begin{align*}
\mathbb{E}[Z_{n,k}^2]&=A_{2k}(k)n^{2k}+A_{2k-1}(k)n^{2k-1}+ \cdots + A_1(k)n+A_0(k),\\
\mathbb{E}[(Z_{n,k})]^2&=B_{2k}(k)n^{2k}+B_{2k-1}(k)n^{2k-1}+ \cdots + B_1(k)n+B_0(k).
\end{align*}

First we will find $A_{2k}(k)$, then show that $A_{2k}(k)=B_{2k}(k).$\\

We observe that $n^{2k}$ appears in the formula of $\mathbb{E}[Z_{n,k}^2]$ given in  \eqref{eqn:2ndmomentZk} only if $t$ is 0, since the only term involving $n$ is the binomial term $\binom{n}{2k-t}.$ The leading term in the formula, which corresponds to $t=0$, is

\begin{equation} \label{zsimp}
\frac{4^k}{(2k)!} \binom{n}{2k} \sum\limits_{s=0}^k \binom{k-s-1/2}{k-s} \binom{s-1/2}{s} \binom{2k}{2k-2s}.
\end{equation} 
Now we use the identities below, which can be found in \cite{gould}, to simplify \eqref{zsimp}.

\begin{equation}\label{id}
\begin{split}
\binom{-1/2}{i} &= (-1)^i \binom{2i}{i} 2^{-2i} \quad \textrm{for all $i \in \mathbb{R}$,} \\
\binom{-r}{i} &= (-1)^i\binom{r+i-1}{i} \quad \quad \textrm{for all $i \in \mathbb{N}$ and $r \in \mathbb{R}$} 
\end{split}
\end{equation}
to obtain
\begin{align}
\frac{4^k}{(2k)!} & \binom{n}{2k} \sum\limits_{s=0}^k \binom{k-s-1/2}{k-s} \binom{s-1/2}{s} \binom{2k}{2k-2s} \\ &= \frac{1}{(2k)!} \binom{n}{2k} \sum_{s=0}^k \binom{2k}{2k-2s} \binom{2k-2s}{k-s} \binom{2s}{s} \notag \\
&= \frac{1}{(2k)!} \binom{n}{2k} \sum_{s=0}^k \binom{2k}{k} \binom{k}{s}^2 \notag\\
&=  \frac{1}{(2k)!} \binom{n}{2k} \binom{2k}{k}^2 \notag \\
&=  \frac{n(n-1)\cdots(n-2k+1)}{(k!)^4}.\label{tis0}
\end{align}
The third equality follows from the fact that $\binom{2k}{k}=\sum_{s=0}^k \binom{k}{s}^2.$ Therefore, by the calculations above, we have $A_{2k}(k)=\frac{1}{(k!)^4}$. The part $(i)$ is proven. \\

For the part $(ii),$ first we find the coefficient $B_{2k}(k)$ in the expansion of $\mathbb{E}[Z_{n,k}]^2.$ We have,
\begin{align}
\mathbb{E}[Z_{n,k}]^2 &= \binom{n}{k}^2 \frac{1}{(k!)^2} \notag \\
&= \frac{n^2(n-1)^2 \cdots (n-k+1)^2}{(k!)^4} \label{B}
\end{align}
It follows from above that $B_{2k}(k)=A_{2k}(k)=\frac{1}{(k!)^4}, $ which implies the variance can be of order $n^{2k-1}$ at most. Next we compare the second coefficients, $A_{2k-1}(k)$ and $B_{2k-1}(k).$\\

In the formula of $\mathbb{E}[Z_{n,k}^2]$, the only terms with $n^{2k-1}$ have $t$ either 0 or 1. We already simplified the sum of terms for which $t=0$ \eqref{tis0}. The terms with $t=1$ in \eqref{eqn:2ndmomentZk} add up to
\begin{align}
 &\frac{1}{(2k-1)!} \binom{n}{2k-1} \sum_{s=0}^{k-1} \binom{k-s-3/2}{k-s-1} \binom{2k-1}{2k-2s-2}  \notag \\
&= \frac{1}{(2k-1)!} \binom{n}{2k-1} \sum_{s=0}^{k-1} 4^s \binom{2k-1}{2k-2s-2} \binom{2k-2s-2}{k-s-1} \notag \\
 &= \frac{1}{(2k-1)!} \binom{n}{2k-1} \sum_{s=0}^{k-1} 4^s \binom{2k-1}{2s+1} \binom{2k-2s-2}{k-s-1}, \label{z2simp}
\end{align}
where the second equality follows from the identities above, \eqref{id}. In order to evaluate the sum we state the following identity, which was proved in \cite{MG}.
\begin{equation} \label{id2}
\sum_{i=0}^m \binom{r+1}{2i+1} \binom{r-2i}{m-i}2^{2i+1} = \binom{2r+2}{2m+1}  \quad \quad \textrm{for all $m \in \mathbb{N}$ and $r \in \mathbb{R}$}. 
\end{equation} 
Applying \eqref{id2} to \eqref{z2simp} by taking $r$ to be $2k-1$ and $m$ to be $k-1$, we obtain
\begin{equation}
\frac{1}{(2k-1)!} \binom{n}{2k-1} \, \frac{1}{2} \, \binom{4k-2}{2k-1}    
\label{tis1}.
 \end{equation} 
So in order to find $A_{2k-1}(k)$, we add the coefficient of $n^{2k-1}$ in \eqref{tis1} to the coefficient of $n^{2k-1}$ in \eqref{tis0}, which is $-\frac{2k^2-k}{(k!)^4}$. Therefore we have,
\begin{equation*}
A_{2k-1}(k)= -\frac{2k^2-k}{(k!)^4} + \frac{1}{2((2k-1)!)^2} \binom{4k-2}{2k-1}.  
\end{equation*}
Next we calculate $B_{2k-1}(k)$ from  \eqref{B} to find the asymptotics of the variance. The coefficient of $n^{2k-1}$ in \eqref{B} is
 \begin{equation*}
B_{2k-1}(k)= -\frac{k^2-k}{(k!)^4}.  
\end{equation*}
Therefore,
\begin{align*}
\text{Var}(Z_{n,k}) &\sim (A_{2k-1}(k) - B_{2k-1}(k)) \, n^{2k-1} \notag \\
&=\left[\frac{1}{2((2k-1)!)^2} \binom{4k-2}{2k-1} - \frac{1}{(k!)^2 ((k-1)!)^2} \right] n^{2k-1} \notag \\
&= \frac{1}{2((2k-1)!)^2}\left[ \binom{4k-2}{2k-1}-2\binom{2k-1}{k}^2 \right]  n^{2k-1}. 
\end{align*}
\end{proof}

\section{Moment asymptotics in random words}\label{sec:momentastmptoticsinwords}
In this section, we focus on random words where the letters are uniformly distributed over a finite alphabet, and study the second moment of the number of weakly increasing subsequences of a given length.

\begin{theorem} (The asymptotics for the second moment and the variance of $Y_{n,k}^{(a)}$) \\ \\
(i.) We have
 \begin{equation*}
\mathbb{E}[(Y_{n,k}^{(a)})^2]\sim \frac{1}{(k!)^2} \frac{1}{a^{2k}} \binom{k+a-1}{k}^2 n^{2k},
\end{equation*}
where $\mathbb{E}[(Y_{n,k}^{(a)})^2]$ is given in \eqref{eqn:2ndmomentYk}.  \\
(ii.) We have  
 \begin{equation*}
\text{Var}(Y_{n,k}^{(a)})\sim C(k,a) \, n^{2k-1}
\end{equation*}
where 
\begin{gather*}
C(k,a)=\frac{1}{a^{2k}}\Bigg[ \frac{a}{(2k-1)!} \left( \sum\limits_{s=0}^{\min(a-1,k-1)} 4^s \frac{(2k+a-s-2)!}{(2s+1)!(a-s-1)!((k-s-1)!)^2} \right) \\ - \frac{1}{((k-1)!)^2} \binom{k+a-1}{k}^2 \Bigg].
\end{gather*}
\end{theorem}

\begin{proof} 
As in the proof of Theorem \ref{asZk}, we write 
\begin{align*}
\mathbb{E}[(Y_{n,k}^{(a)})^2]&=A_{2k}(k,a)n^{2k}+A_{2k-1}(k,a)n^{2k-1}+ \cdots + A_1(k,a)n+A_0(k,a),\\
\mathbb{E}[(Y_{n,k}^{(a)})]^2&=B_{2k}(k,a)n^{2k}+B_{2k-1}(k,a)n^{2k-1}+ \cdots + B_1(k,a)n+B_0(k,a).
\end{align*}

Again as in the proof of Theorem \ref{asZk}, first we evaluate the sum of terms in the equation for $\mathbb{E}[(Y_{n,k}^{(a)})^2]$, \eqref{eqn:2ndmomentYk}, corresponding to $t=0$ and then to $t=1.$ For $t=0$, we have
\begin{align}
& \frac{4^k}{a^{2k}} \binom{n}{2k} \sum\limits_{s=0}^{\min(a-1,k)} \binom{k-s-1/2}{k-s} \binom{s-1/2}{s} \binom{2k-s+a-1}{2s} \binom{2k-3s+a-1}{-s+a-1} \notag \\ &= \frac{1}{a^{2k}} \binom{n}{2k} \sum\limits_{s=0}^{\min(a-1,k)} \binom{2k-2s}{k-s} \binom{2s}{s} \binom{2k-s+a-1}{2s} \binom{2k-3s+a-1}{-s+a-1}  \notag \\
 &= \frac{1}{a^{2k}} \binom{n}{2k} \sum\limits_{s=0}^{\min(a-1,k)} \frac{(2k+a-1-s)!}{(a-1-s)!((k-s)!)^2(s!)^2}  \notag \\
&=  \frac{1}{a^{2k}} \binom{n}{2k}\binom{2k}{k} \sum\limits_{s=0}^{\min(a-1,k)} \binom{2k+a-1-s}{2k}  \binom{k}{s}^2 \notag \\ 
&=  \frac{1}{a^{2k}} \binom{n}{2k}\binom{2k}{k} \binom{k+a-1}{k}^2,  \label{Tis0}
\end{align}
where the first equality follows from the identities \eqref{id} and the last equality follows from the combinatorial fact \eqref{id3} below in \cite{R}.

\begin{equation} \label{id3}
\sum_{i=0}^{\min(m,l)} \binom{m+2l-i}{2l} \binom{l}{i}^2  = \binom{m+l}{l}^2  \quad \quad \textrm{for all $m,l \in \mathbb{N}$}. 
\end{equation}
Therefore, $A_{2k}(k,a),$ the coefficient of $n^{2k}$ above, is $\frac{1}{(k!)^2} \frac{1}{a^{2k}} \binom{k+a-1}{k}^2 $. The part $(i)$ is proven. \\

Second, we deal with the expansion of $\mathbb{E}[(Y_{n,k}^{(a)})]^2$.  Observe that 
\begin{align*}
\mathbb{E}[(Y_{n,k}^{(a)})]^2&=\binom{n}{k}^2 \binom{k+a-1}{k}^2 \frac{1}{a^{2k}}  \notag \\
&= \binom{k+a-1}{k}^2 \frac{1}{a^{2k}} \frac{n^2(n-1)^2 \cdots (n-k+1)^2}{(k!)^2} 
\end{align*}
So we have,
\begin{align}
B_{2k}(k,a)&=\frac{1}{(k!)^2} \frac{1}{a^{2k}} \binom{k+a-1}{k}^2,  \notag \\
B_{2k-1}(k,a) &= -\frac{k^2-k}{(k!)^2} \binom{k+a-1}{k}^2 \frac{1}{a^{2k}}. \notag 
\end{align}
Finally $A_{2k-1}(k,a)$ is to be found. Following the proof for the random permutation case, we write the sum of terms for $t=1$ in \eqref{eqn:2ndmomentYk},

\medmuskip=0mu
\thinmuskip=0mu
\thickmuskip=0mu

\begin{align*}
&= \frac{4^{k-1}}{a^{2k-1}} \binom{n}{2k-1} \sum\limits_{s=0}^{\min(a-1,k-1)} \binom{k-s-3/2}{k-s-1} \binom{2k+a-s-2}{2s+1}  \binom{2k-3s+a-3}{-s+a-1} \notag \\
&= \frac{1}{a^{2k-1}} \binom{n}{2k-1} \sum\limits_{s=0}^{\min(a-1,k-1)} 4^s \binom{2k+a-s-2}{2s+1} \binom{2k+a-3s-3}{a-s-1} \binom{2k-2s-2}{k-a-s-1} \notag \\
&= \frac{1}{a^{2k-1}} \binom{n}{2k-1} \sum\limits_{s=0}^{\min(a-1,k-1)} 4^s \frac{(2k+a-s-2)!}{(2s+1)!(a-s-1)!((k-s-1)!)^2}  
\end{align*}
\thinmuskip=3mu
\medmuskip=4mu plus 2mu minus 4mu
\thickmuskip=5mu plus 5mu
Adding the coefficient of $n^{2k-1}$ for the terms having $t=0$, which can be found in \eqref{Tis0} similar to previous case, we have
\begin{equation*}
\begin{split}
A_{2k-1}(k,a)= & - \frac{1}{a^{2k}} \frac{2k^2-k}{(k!)^2} \binom{k+a-1}{k}^2 + \frac{1}{a^{2k-1}} \frac{1}{(2k-1)!} \\
 & \times \sum\limits_{s=0}^{\min(a-1,k-1)} 4^s \frac{(2k+a-s-2)!}{(2s+1)!(a-s-1)!((k-s-1)!)^2}
\end{split}
\end{equation*}
Therefore,
\begin{align*}
\text{Var}(Y_{n,k}^{(a)}) &\sim (A_{2k-1}(k,a) - B_{2k-1}(k,a)) \, n^{2k-1}  \\
&= \frac{1}{a^{2k}}\Bigg[\frac{a}{(2k-1)!} \left(\sum\limits_{s=0}^{\min(a-1,k-1)} 4^s \frac{(2k+a-s-2)!}{(2s+1)!(a-s-1)!((k-s-1)!)^2} \right) \\ & \qquad \quad - \frac{\binom{k+a-1}{k}^2}{((k-1)!)^2} \Bigg]  n^{2k-1}&. 
\end{align*}

 \end{proof}

\section{Comparison between words and permutations}\label{sec:comparison}
Next, we explore connections between the
statistics $Y_{n,k}^{\mathbf{(p)}}$ and $Z_{n,k}$ (and,  similarly for $Y_n^{\mathbf{p}}$ and
$Z_n$). Considering the uniform case,  it is intuitively  clear that for large $a$, these statistics
should be close to each other in distribution as the possibility of
having same numbers disappears for the random word case. We
formalize this below by comparing the tail probabilities
corresponding the uniformly random permutation case and random word
case where the letters are not necessarily equally likely.
\begin{theorem}\label{discrepgen} Let $d_{TV}$ denote the total variation distance. \\ 
 (i.)We have
\begin{equation}\label{TVboundforbiasedcase}
    d_{TV}(Y_{n,k}^{\mathbf{(p)}}, Z_{n,k}) \leq \binom{n}{2} \sum_{i=1}^a p_i^2.
\end{equation} When $p$ is the uniform distribution over $[a]$, $a \in \mathbb{N}$,  the bound in
\eqref{TVboundforbiasedcase} can be improved to
\begin{equation*}
    d_{TV}(Y_{n,k}^{\mathbf{(p)}}, Z_{n,k}) \leq 1- \frac{a!}{(a-n)!}
    \frac{1}{a^n}, \qquad a \geq n.
\end{equation*}

\noindent
(ii.) Further, for any $z \in \mathbb{R}$, $a \geq n \geq 1$, we
have
$$\mathbb{P}(Y_{n,k}^{(a)} \geq z) \in \left(\mathbb{P}(Z_{n,k} \geq z),
\mathbb{P}(Z_{n,k} \geq z) + 1- \frac{a!}{(a-n)!}
\frac{1}{a^n}\right].$$

\noindent
(iii.) For fixed $n$ and $k,$ $Y_{n,k}^{(a)} \longrightarrow_d Z_{n,k}$
as $a \rightarrow \infty$.

\noindent
(iv.) The results in (i.), (ii.) and (iii.) also hold when $Y_{n,k}^{\mathbf{(p)}}$ and
$Y_{n,k}^{(a)}$ are replaced by $Y_{n}^{(\mathbf{p})}$ and $Y_{n}^a$ respectively.

\end{theorem}

Proof of the second part of Theorem \ref{discrepgen} will require a
stochastic dominance relation between random word and random
permutation statistics. 

\begin{proposition}\label{thm:stochasticdominance}
(i.) For any $a, b \in \mathbb{N}$ with $a \leq b$, we have
$$Y_{n,k}^b \leq_s Y_{n,k}^{(a)}.$$
(ii.) For any $a \geq 1$, $$Z_{n,k} \leq_s Y_{n,k}^{(a)}.$$
\end{proposition}

\noindent \textbf{Proof of  Proposition \ref{thm:stochasticdominance}.}
 \textbf{(i.)} The idea is to find a coupling $(\xi_1,\xi_2)$ of $Y_{n,k}^{(a)}$ and
$Y_{n,k}^{(b)}$ so that $Z_2 \leq Z_1$ almost surely. Let $d := ab$ and
$V_1,\ldots,V_n$ be independent uniformly distributed random
variable over $[d]$. For $j=1,\ldots,n$, let
$$
U_1^j =
\begin{cases}
1, & \text{if } \; V_j =1,\ldots,b\\
2, & \text{if } \;  V_j= b+1,\ldots,2b \\
\ldots   & \ldots\\
a, & \text{if } \;  V_j=(a-1)b+1,\ldots,ab.
\end{cases}
$$
and
$$
U_2^j =
\begin{cases}
1, & \text{if } \; V_j =1,\ldots,a\\
2, & \text{if } \;  V_j= a+1,\ldots,2a \\
\ldots   & \ldots\\
b, & \text{if } \;  V_j=(b-1)a+1,\ldots,ab.
\end{cases}
$$
Then $U_1^1,\ldots,U_1^n$ are independent uniformly distributed over
$[a]$, and $U_2^1,\ldots,U_2^n$ are independent uniformly
distributed over $[b]$. Set 
\begin{align*}
\xi_1 =& \sum_{1\leq i_1 < \cdots < i_k
\leq n} \mathbf{1}(U_1^{i_1}\leq \cdots \leq U_1^{i_k}) \qquad
\text{and} \\
\xi_2 =& \sum_{1\leq i_1 < \cdots < i_k \leq n}
\mathbf{1}(U_2^{i_1}\leq \cdots \leq U_2^{i_k}).
\end{align*}
Clearly,
$(\xi_1,\xi_2)$ is a coupling of $Y_{n,k}^{(a)}$ and $Y_{n,k}^b$, and
further we have $\xi_2 \leq \xi_1$ by construction. The result follows.

\bigskip

\noindent \textbf{(ii.)} Let us assume for a contradiction that $Y_{n,k}^{(a)} <_s Z_{n,k}$
for some $a \in \mathbb{Z}$. Then for any $b \geq a$, we have
$$ Y_{n,k}^{(b)} \leq_s Y_{n,k}^{(a)}  <_s Z_{n,k},$$ where for the first inequality we used part $i.$.
Now taking the limit as $b \rightarrow \infty$, and using the first
part of Theorem \ref{discrepgen}, we arrive at the conclusion that
$Z_{n,k} <_s Z_{n,k}$ which is a contradiction. 
\qed

\begin{remark}
 (i.) The stochastic dominance relation in Proposition
\ref{thm:stochasticdominance} is actually slightly more general. To give
another example, letting $\mathcal{S}_1=\{(i_1,\ldots,i_k): i_j \in
[n], j=1,\ldots,k \; \text{and} \; 1\leq i_1<\cdots<i_k \leq n\}$,
such a dominance result would also hold for the statistic
$$\sum_{(i_1,\ldots,i_k) \in \mathcal{S}_1} \mathbf{1}(X_{i_1} \Delta_1 X_{i_2} \Delta_2 \cdots \Delta_{k-1} X_{i_k}),$$ where
$\Delta_i$ can be any of $\geq, \leq, =$ for each $i =
1,\ldots,k-1$.

(ii.) Focusing on the case $k=2$, the means and variances of
$Y_{n,2}^{(2)}, Y_{n,2}^{(a)}, Z_{n,2}$ are all of the same order. So thanks
to stochastic dominance result in Proposition
\ref{thm:stochasticdominance}, it would not be surprising to obtain
the asymptotic normality of $Y_{n,k}^{(a)}$, $a\geq 3$ by the
corresponding results for $Y_{n,k}^2$ and $Z_{n,k}$. This is
especially interesting as in some problems it can be easier to prove
the results for both binary random words and uniformly random
permutation cases, but not for random words with a larger alphabet.
Of course one may question the $k \geq 3$ case in a similar way.
\end{remark}

Now we are ready to give the proof of Theorem \ref{discrepgen}.

\bigskip

\noindent \textbf{Proof of Theorem \ref{discrepgen}.} \textbf{(i.)}
Let $X_1,\ldots,X_n$ be independent random variables with
$\mathbb{P}(X_i=j)=p_j$ for $j=1,\ldots,a$, $a \geq 2.$ Let $Y_{n,k}^{\mathbf{(p)}}$
be the number of weakly increasing subsequences of $X_1,\ldots,X_n$ of
length $k.$ Also define $T$ to be the number of different elements
in the sequence $X_1,\ldots,X_n$. Then for any $A\subset \mathbb{R}$,
we have
\begin{align}\label{unifriff}
  \mathbb{P}(Y_{n,k}^{\mathbf{(p)}}\in A ) =& \mathbb{P}(Y_{n,k}^{\mathbf{(p)}} \in A , T=n) +  \mathbb{P}(Y_{n,k}^{\mathbf{(p)}} \in A, T<n ) \nonumber\\
   =& \mathbb{P}(Y_{n,k}^{\mathbf{(p)}} \in A|T=n) \mathbb{P}(T=n) +\mathbb{P}(Y_{n,k}^{\mathbf{(p)}} \in
    A, T<n) \nonumber\\
    \leq& \mathbb{P}(Z_{n,k} \in A) +\mathbb{P}(Y_{n,k}^{\mathbf{(p)}} \in A,
    T<n)
\end{align}
where (\ref{unifriff}) follows by observing $\mathbb{P}(Y_{n,k}^{\mathbf{(p)}}
\in A|T=n)=\mathbb{P}(Z_{n,k} \in A)$. This yields
\begin{align}\label{totvar1}
  \mathbb{P}(Y_{n,k}^{\mathbf{(p)}} \in A) -    \mathbb{P}(Z_{n,k} \in A) \leq   \mathbb{P}(Y_{n,k}^{\mathbf{(p)}} \in A,
  T<n) \leq \mathbb{P}(T<n).
\end{align}
Similarly, we have
\begin{align*}
  \mathbb{P}(Z_{n,k} \in A) =& \mathbb{P}(Z_{n,k} \in A)\mathbb{P}(T=n) + \mathbb{P}(Z_{n,k} \in A)\mathbb{P}(T<n) \\
  =& \mathbb{P}(Y_{n,k}^{\mathbf{(p)}} \in A | T=n) \mathbb{P}(T=n) + \mathbb{P}(Z_{n,k} \in A)\mathbb{P}(T<n)\\
  \leq& \mathbb{P}(Y_{n,k}^{\mathbf{(p)}} \in A) + \mathbb{P}(T<n),
\end{align*}
implying
\begin{equation}\label{totvar2}
\mathbb{P}(Z_{n,k} \in A) -\mathbb{P}(Y_{n,k}^{\mathbf{(p)}} \in A) \leq
\mathbb{P}(T<n).
\end{equation}
Hence combining (\ref{totvar1}) and (\ref{totvar2}), for $a \geq n$,
we have
\begin{align*}
  d_{TV}(Y_{n,k}^{\mathbf{(p)}}, Z_{n,k})\leq \mathbb{P}(T<n) =&
 \mathbb{P} \left( \bigcup_{i \neq j} \{X_i = X_j\} \right) \\
   \leq &\sum_{i \neq j} \mathbb{P}(X_i = X_j) \\
   = & \binom{n}{2} \sum_{i = 1}^a p_i^2,
\end{align*}
which proves the first claim.

The  estimate for the uniform case is similar with the only
difference being at the last step where this time we have
\begin{align*}
  d_{TV}(Y_{n,k}^{(a)}, Z_{n,k})\leq \mathbb{P}(T<n) =
 \mathbb{P} \left( \bigcup_{i \neq j} \{X_i = X_j\} \right)
   =& 1 - \mathbb{P}
\left(\bigcap_{i \neq j} \{X_i \neq X_j\} \right) \\
=& 1 - \frac{\binom{a}{n} n!}{a^n}  \\
=& 1- \frac{a!}{(a-n)!} \frac{1}{a^n}
\end{align*}

\noindent \textbf{(ii.)} We know from Theorem \ref{discrepgen} that the inequality
$$|P(Y_{n,k}^{(a)} \geq z) - P(Z_{n,k} \geq z)| \leq
1- \frac{a!}{(a-n)!} \frac{1}{a^n}$$ holds for any $z \in
\mathbb{R}$ since the total variation distance provides an upper
bound on the Kolmogorov distance. Also by the stochastic dominance
result in Proposition \ref{thm:stochasticdominance}, we have
$$P(Y_{n,k}^{(a)} \geq z) \geq P(Z_{n,k} \geq z)$$ for any $a \geq 1$.  Combining these two
observations immediately reveal the required result.

\bigskip

\noindent \textbf{(iii.)} This follows from the fact that convergence in total
variation distance implies convergence in distribution.

\bigskip

\noindent \textbf{(iv.)} We just need to replace $Y_{n,k}^{\mathbf{(p)}}$ and $Y_{n,k}^{(a)}$ by $Y_{n}^{(\mathbf{p})}$
and $Y_n^{(a)}$ in above proof.  

\hfill \qed

\section{Two applications}\label{sec:applications}

\subsection{A sequence comparison statistic of Steele}\label{sec:Steele}

\medmuskip=0mu
\thinmuskip=0mu
\thickmuskip=0mu

For two permutations  $\pi$ and $\rho$  in $S_n$, we define $$V_n = V_n(\pi, \rho) = \sum_{k=1}^n \sum_{1 \leq i_1 < \cdots < i_k \leq n} \sum_{1 \leq j_1 < \cdots < j_k \leq n}  \mathbf{1} (\pi(i_1) = \rho(j_1),\ldots, \pi(i_k) = \rho(j_k)).$$

\thinmuskip=3mu
\medmuskip=4mu plus 2mu minus 4mu
\thickmuskip=5mu plus 5mu

This corresponds to a similarity measure of M. Steele, first introduced in \cite{Steele} in terms of random words instead of  permutations. In this setting, we have the following result regarding first two moments of  $V_n$. The proof below turns the similarity measure problem into an increasing subsequence problem and uses the well-known results on uniformly random permutations. This approach was previously made use in \cite{HI} in order to understand the length of longest common subsequences of two independent random permutations. 

\begin{theorem}\label{thm:twomomentsVn}
Let $\pi$ be a uniformly random permutation in $S_n$ and $\rho$ be an independent random permutation with \emph{any} distribution. Then we have

(i.) We have $$\mathbb{E}[V_n] = \sum_{k=1}^n \binom{n}{k}^2 \frac{(n-k)!}{n!}, \, \text{and} \quad \mathbb{E}[V_n] \sim \frac{1}{2 \sqrt{\pi e} n^{1/4}} e^{2 n^{1/2}}, \, \text{as} \, n \rightarrow \infty.$$

(ii.)  We have $$Var(V_n) = \sum_{k+ l \leq n} 4^l \frac{1}{(k+l)!} \binom{n}{k+l} \binom{(k+1)/2 + l -1}{l} -  \left(\sum_{k=1}^n \binom{n}{k}^2 \frac{(n-k)!}{n!}\right)^2.$$
 Furthermore,  $$Var(V_n) \sim \frac{C}{n^{1/4}} e^{2 \sqrt{2 + \sqrt{5}} n^{1/2}}, \quad \text{as} \quad n \rightarrow \infty,$$ where $C$ is a constant. 

(iii.) In particular, $$\frac{V_n - \mathbb{E}[V_n]}{\sqrt{Var(V_n)}} \longrightarrow_{\mathbb{P}} 0, \quad \text{as} \quad n \rightarrow \infty. $$
\end{theorem}
 
\begin{remark}
Let us emphasize that the results of Theorem \ref{thm:twomomentsVn}  hold true  for any distribution on $\rho$. In particular, $\rho$ can be a fixed permutation. 
\end{remark}

\begin{proof}
 To deal with  (i.) and (ii.) we will  turn the problem into a problem of increasing subsequences and use the  corresponding result of \cite{LP1981}.  For this purpose, let us   introduce some notation for convenience. Define 
 \begin{multline*}
 \mathcal{S}_{n, \mathbf{i}, \mathbf{j}} = \{((k,(i_1,\ldots,i_k),(j_1,\ldots,j_k)): k \in [n], \\ 1 \leq i_1<  i_2 <\cdots < i_k \leq n, 1 \leq j_1<  j_2 <\cdots < j_k \leq n\}
\end{multline*} 
 and $$\mathcal{S}_{n, \mathbf{l}} =  \{((k,(l_1,\ldots,l_k)): k \in [n], 1 \leq l_1< l_2 <\cdots < l_k \leq n\}.$$ 
Note in particular that $V_n = \sum_{\mathcal{S}_{n, \mathbf{i}, \mathbf{j}}}  \mathbf{1} (\pi(i_1) = \rho(j_1),\ldots, \pi(i_k) = \rho(j_k))$. 

Next, let $\pi$ and $\rho$ be as in statement of the result, and $\tau$ be another uniformly random permutation in $S_n$.  We  claim that 
$$V_n =_d \sum_{\mathcal{S}_{n, \mathbf{l}}} \mathbf{1}(\tau(l_1) < \cdots < \tau(l_k)).$$
To prove this, first let $id$ be the identity permutation, and observe that 
\begin{align*}
\sum_{\mathcal{S}_{n, \mathbf{i}, \mathbf{j}}}\mathbf{1}(\pi(i_1) = id(j_1), \cdots, \pi(i_k) = id(j_k)) =& \sum_{\mathcal{S}_{n, \mathbf{l}}} \mathbf{1}(\pi(i_1) < \cdots <\pi(i_k)) \\ 
=_d& \sum_{\mathcal{S}_{n, \mathbf{l}}} \mathbf{1}(\tau(i_1) < \cdots <\tau(i_k)).
\end{align*}
Next, let $\gamma$ be any other fixed permutation in $S_n$, and note that $\pi \gamma$ is still a uniformly random permutation. Then we have 
\begin{align*}
&\sum_{\mathcal{S}_{n, \mathbf{i}, \mathbf{j}}}\mathbf{1}(\pi(i_1) = \gamma(j_1), \cdots, \pi(i_k) = \gamma(j_k)) \\ =_d& \sum_{\mathcal{S}_{n, \mathbf{i}, \mathbf{j}}}\mathbf{1}(\pi(\gamma(i_1)) = \gamma(j_1), \cdots, \pi(\gamma(i_k)) = \gamma(j_k)) \\
=_d& \sum_{\mathcal{S}_{n, \mathbf{i}, \mathbf{j}}}\mathbf{1}(\pi(i_1) = j_1, \cdots, \pi(i_k) = j_k) \\
=_d& \sum_{\mathcal{S}_{n, \mathbf{l}}} \mathbf{1}(\tau(i_1) < \cdots <\tau(i_k)), 
\end{align*}
as above. 

Finally, recalling that $\rho$ is any random permutation, for any $x \in \mathbb{R}$, we have 
\begin{align*}
\lefteqn{\mathbb{P}\left(\sum_{\mathcal{S}_{n, \mathbf{i}, \mathbf{j}}} \mathbf{1} (\pi(i_1) = \rho(j_1),\ldots, \pi(i_k) = \rho(j_k)) \leq x \right) }  \\ =&
 \frac{1}{n!} \sum_{\gamma \in S_n} \mathbb{P}\left(\sum_{\mathcal{S}_{n, \mathbf{i}, \mathbf{j}}} \mathbf{1} (\pi(i_1) = \gamma(j_1),\ldots, \pi(i_k) = \gamma(j_k)) \leq x \big| \rho = \gamma \right)
  \\ =& \frac{1}{n!} \sum_{\gamma \in S_n} \mathbb{P} \left(\sum_{\mathcal{S}_{n, \mathbf{l}}} \mathbf{1}(\pi(i_1) < \cdots \pi(i_k)) \leq x  \right)
   \\ =& \mathbb{P} \left(\sum_{\mathcal{S}_{n, \mathbf{l}}} \mathbf{1}(\tau(i_1) < \cdots \tau(i_k)) \leq x \right).
\end{align*}
That is, $V_n =_d \sum_{\mathcal{S}_{n, \mathbf{l}}} \mathbf{1}(\tau(l_1) < \cdots < \tau(l_k)),$ as claimed and so each claim in first two parts follow from \cite{LP1981}.

(iii.) Let $\epsilon > 0$. Then, using Markov's inequality $\mathbb{P} \left(\left|\frac{V_n - \mathbb{E}[V_n]}{\sqrt{Var(V_n)}} \right| > \epsilon \right) \leq
 \mathbb{P} \left(\left|\frac{V_n - \mathbb{E}[V_n]}{\sqrt{Var(V_n)}} \right| > \epsilon \right) \leq  
 \frac{2 \mathbb{E}[U_n]}{\epsilon \sqrt{Var(U_n)}} \longrightarrow 0,$  as $n \rightarrow \infty$. So, the result follows since convergence in probability implies convergence in distribution. 
\end{proof}

\subsection{Increasing subsequences in riffle shuffles}\label{sec:riffle}

We conclude  the paper with a discussion of  a question of Fulman on the
asymptotic distribution of the number of inversions in riffle
shuffles \cite{fulman}. Indeed, throughout the way we are able to  have the
chance to analyze the number of increasing (or decreasing) subsequences of a given length  in this
shuffling scheme.

In a standard riffle shuffle, one first cuts the deck into
two piles and then riffles the piles together; i.e., drops the
cards from the bottom of each pile to form a new pile. See
\cite{diac} and \cite{fulman} for a detailed account of riffle
shuffles. Following \cite{fulman}, a formal definition of riffle shuffles can be given 
as follows:   Cut the $n$ card deck into $a$
piles by picking pile sizes
  according to the $mult(a;\mathbf{p})$ distribution, where $\mathbf{p}=(p_1,\ldots,p_a)$.
That is, choose $b_1,\ldots,b_a$ with probability
$\binom{n}{b_1,\ldots,b_a} \Pi_{i=1}^a p_i^{b_i}.$ Then choose
uniformly one of the $\binom{n}{b_1,\ldots,b_a}$ ways of interleaving
the packets, leaving the cards in each pile in their original order.
The resulting probability distribution on $S_n$ is called as the
\emph{$\mathbf{p}$-shuffle distribution} and is denoted by $P_{n,a,\mathbf{p}}$. When
$p$ is the uniform distribution, we  write $P_{n,a}$ instead and
call the resulting distribution an \emph{$a$-shuffle
distribution}.

\bigskip

The following provides an alternative description of riffle shuffles which will be useful for our purposes.

\noindent \emph{Alternative description} (Inverse $p$-shuffles) : The inverse
of a biased $a$-shuffle has the following description. Assign
independent random digits from
  $\{1,\ldots,a\}$ to each card with distribution $\textbf{p}=(p_1,\ldots,p_a)$. Then sort according to digit,
  preserving relative order for cards with the same digit.

In other words, if $\sigma$ is generated according to Description 2,
then $\sigma^{-1} \sim P_{n,a, \mathbf{p}}$. Here is the central limit theorem for the number of inversions in riffle shuffles. 
\begin{theorem}\label{thm:invinriffle}
Let $\rho_{n,a}$ be a random permutation with distribution $P_{n,a}$
 with $a\geq 2.$ Then
$$\frac{inv(\rho_{n,a})-\frac{n(n-1)}{4} \frac{a-1}{a}}{\sqrt{n}(n-1)
\sqrt{\frac{a^2-1}{36a^2}}} \longrightarrow_d \mathcal{G}$$ as $n
\rightarrow \infty$.
\end{theorem}

\begin{proof} Let
$\rho_{n,a}$ be a random permutation with distribution $P_{n,a}$
which is  generated via an inverse shuffle with the random word
$X=(X_1,\ldots,X_n)$. Noting that 
 observe that
$$\rho_{n,a}(i)= |\{j:X_j < X_i\}|+|\{j\leq i : X_j = X_i\}|,$$
for $i,k \in [n]$, we have $\rho_{n,a}(i)>\rho_{n,a}(k)$ if and
only if
\begin{gather*}
|\{j:X_j < X_i\}|+|\{j : j\leq i , X_j = X_i\}| \\ >|\{j:X_j <
X_{k}\}|+|\{j : j\leq k, X_j = X_{k}\}|.
\end{gather*}
Using this for the case
$i<k$, we conclude that $$\rho_{n,a}(i)> \rho_{n,a}(k) \quad \text{if
and only if} \quad X_i > X_{k}.$$ Therefore,
$$Inv(\rho_{n,a})=_d Inv(X),$$ and the result follows from Corollary \ref{cor:inversionsinwords}. 
\end{proof}

\appendix

\section{Appendix}

\noindent \textbf{Proof of Proposition \ref{propgen1}.} \label{pfgen1}
In a simpler notation, the series we have in the proposition is
\begin{equation}
\sum\limits_{i,j,k \geq 0} {i+j \choose j} {j+k \choose k }{k+i\choose k}x^{i} y^{j} z^{k}. \label{binomcyc1}
\end{equation}

The generating functions for the cycles of binomial coefficients,
\begin{equation*}
\binom{i_1+i_2}{i_2}\binom{i_2+i_3}{i_3} \cdots \binom{i_n+i_1}{i_1},
\end{equation*}
can be found in the paper by Carlitz \cite{carl}. In our case, when the length of the cycle is $3$, the generating function was shown to be
\begin{equation*}
[(1-x-y-z)^2-4xyz]^{-1/2}.
\end{equation*}

Let us find the coefficient of $x^{n}y^{n}z^m$ in $[(1-x-y-z)^2-4xyz]^{t}$ where $n,m \in \mathbb{N} \cup \{0\}$ and $r \in \mathbb{R}.$ We have the binomial expansion,

\begin{equation*}
[(1-x-y-z)^2-4xyz]^{r} = \sum\limits_{s \geq 0} \binom{r}{s}(-4xyz)^s ((1-(x+y+z))^2)^{r-s}.
\end{equation*}  
We can figure out the coefficient from the binomial expansion by an easy combinatorial argument,
\begin{gather*}
[x^ny^nz^m][(1-x-y-z)^2-4xyz]^{r} = \\ \sum\limits_{s=0}^{\min \{n,m\}} (-4)^s \binom{r}{s} \binom{2r-2s}{2n+m-3s} \binom{2n+m-3s}{m-s}\binom{2n-2s}{n-s}.
\end{gather*}
We can write the coefficient in positive terms by the identities \eqref{id} to have, 
\begin{equation*}
\sum\limits_{s=0}^{\mathclap{\min\{n,m\}}} 4^n  \binom{n-s-1/2}{n-s} \binom{s-r-1}{s} \binom{2n+m-2s-2r-1}{s-2r-1} \binom{2n+m-3s}{m-s}.
\end{equation*}
Finally replace $n,m$ and $r$ by $k-t, a-1$ and $-(t+1)/2$ respectively to conclude the proof of the proposition.
\qed
\\

\noindent \textbf{Proof of Proposition \ref{propgen2}.} \label{pfgen2}
We have a simpler case of Proposition \ref{propgen1}. Consider the series
\begin{equation*}
\sum\limits_{i,j \geq 0} {i+j \choose j} {j+i \choose i }x^{i} y^{j},
\end{equation*}
which is a cycle of binomial coefficients that has length $2$, cf.\eqref{binomcyc1}. The corresponding generating function can be found in \cite{carl}, which is
\begin{equation*}
[(1-x-y)^2-4xy]^{-1/2}.
\end{equation*}
Then we follow exactly the same steps in the Proposition  \ref{propgen2} given in \ref{pfgen1}  above to arrive at the desired result.
\qed

\end{document}